\documentclass{amsart}
\usepackage[margin=3cm]{geometry}
\usepackage{amsfonts,amsmath,amssymb,amscd,graphicx}
\usepackage[colorlinks]{hyperref}
\usepackage{mathrsfs}

\newtheorem{theorem}{Theorem}[section]
\newtheorem{lemma}[theorem]{Lemma}
\newtheorem{corollary}[theorem]{Corollary}
\newtheorem{proposition}[theorem]{Proposition}

\theoremstyle{definition}

\newtheorem{example}[theorem]{Example}

\theoremstyle{remark}
\newtheorem{remark}[theorem]{Remark}

\numberwithin{equation}{section}

\newcommand{\beq}{\begin{eqnarray}}
\newcommand{\eeq}{\end{eqnarray}}
\newcommand{\beqe}{\begin{eqnarray*}}
	\newcommand{\eeqe}{\end{eqnarray*}}

\DeclareMathOperator{\density}{volfrac}
\DeclareMathOperator{\codim}{codim}

\DeclareMathOperator{\vol}{vol}

\DeclareMathOperator{\spec}{spec}

\DeclareMathOperator{\dist}{dist}

\DeclareMathOperator{\vect}{Vect}
\DeclareMathOperator{\sym}{Sym}

\DeclareMathOperator{\ricci}{Ric}
\DeclareMathOperator{\scal}{Scal}
\DeclareMathOperator{\hbc}{HBC}
\DeclareMathOperator{\hc}{HSC}

\DeclareMathOperator{\tr}{Tr}
\DeclareMathOperator{\mat}{Mat}
\newcommand{\cov}{\ensuremath \mathrm{Cov}}
\newcommand{\var}{\ensuremath \mathrm{Var}}
\newcommand{\id}{\ensuremath \mathrm{Id}}

\newcommand{\End}{\ensuremath \mathrm{End}}
\newcommand{\Id}{\ensuremath \mathrm{Id}}

\newcommand{\Nn}{\mathbb{N}}
\newcommand{\R}{\mathbb{R}}

\newcommand{\C}{\mathbb{C}}

\let\epsilon=\varepsilon

\begin{document}

\title[On the curvatures of random complex submanifolds]{On the curvatures of \\ random complex submanifolds}
\author{Michele Ancona }
\address{Laboratoire J.A. Dieudonn\'e,
	UMR CNRS 7351,
	Universit\'e C\^ote d'Azur, Parc Valrose,
	06108 Nice, Cedex 2, France}
\email{michele.ancona@unice.fr}

\author{Damien Gayet}
\address{ Univ. Grenoble Alpes, Institut Fourier,
	F-38000 Grenoble, France,
	CNRS UMR 5208,
	CNRS, IF, F-38000 Grenoble, France}
\email{damien.gayet@univ-grenoble-alpes.fr}
\thanks{The research leading to these results has received funding from the French Agence nationale de la recherche  ANR-20-CE40-0017 (Adyct).}

\begin{abstract} 
	For any integers $n\geq 2$ and $1\leq r\leq n-1$ satisfying $3r\geq 2n-1$, we show that the expected volume fraction 
	of a random degree $d$ complex submanifold of $\C\mathbb{P}^n$ of codimension $r$  where the bisectional holomorphic curvature (for the induced ambient metric) is negative tends to one when $d$ goes to infinity. Here, the probability measure is the natural one associated with the Fubini--Study metric. 
	We provide similar estimates for the holomorphic sectional curvature, the Ricci curvature, and the scalar curvature. Our results hold more generally for random submanifolds within any complex projective manifold. 

\end{abstract}

\maketitle

\section{Introduction}
Let $M$ be a K\"ahler manifold of dimension $n$ and denote by $J$ the complex structure of $M$. Let $g$ be the K\"ahler metric on $M$ and $R:=R^M$ be the associated Riemannian tensor curvature. Recall that for any $x\in M$ and two non-zeros tangent vectors $X,Y\in T_xM$, the {\it holomorphic bisectional curvature} at $x$ in the directions $(X,Y)$ equals~\cite[Note 23]{kobayashi}:
$$\hbc_{(M,g)}(X,Y) = \frac{R(X, JX, Y , JY )}{\|X\|^2_g\|Y\|^2_g}.$$
We will sometimes write $\hbc_{M}(X,Y)$ when the metric we choose is clear. when the metric we choose is clear. The holomorphic bisectional curvature depends only on the two complex lines defined by $X$ and $Y$.
 Note also that the holomorphic bisectional curvature can be recovered from the Riemannian sectional curvature using the formula
 $$ R(X,JX, Y, JY) = R(X, Y, X, Y) + R(X, JY, X, JY).$$
  Having a \emph{positive} holomorphic bisectional curvature is a very strong constraint, see~\cite{goldberg1967holomorphic}. In particular, any  such compact K\"ahler manifold is biholomorphic to $\C\mathbb{P}^n$~\cite{siu1980compact}. On the other hand, any complex submanifold in $\C^n$, or in a complex flat torus, has  non-positive bisectional curvature, see Proposition~\ref{proposition2} below. 
  For $x\in M$ and $X\in T_xM$, the \emph{holomorphic sectional curvature} of $g$ at $x$ is defined by  
  $$\hc_M(X)= \hbc_M(X,X),$$ and the \emph{Ricci curvature} by  $$ \ricci_M(X)= \sum_{i=1}^n \hbc_M(X,e_i),$$
  where $\{e_i,Je_i\}_{i=1,\cdots, n}$ is any orthonormal basis of $(T_xM,g)$. Finally, the \emph{scalar curvature} equals
  $$ \scal_M= \sum_{i,j=1}^n \hbc_M(e_i,e_j).$$
For complex curves, all these curvatures coincide. In general, the holomorphic bisectional curvature determines the holomorphic sectional curvature,  the Ricci curvature, and the scalar curvature.    Note also that the Ricci and scalar curvatures can be defined on a general Riemannian manifold, not necessarily K\"ahler.  
   
\subsection{Probabilistic setting} The goal of this paper is to study these four curvatures for \emph{random} complex submanifolds $Z$ of a complex projective manifold, equipped with the restriction $g_{|Z}$ of a fixed ambient K\"ahler metric $g$. We now introduce our setting, which has become quite standard since~\cite{SZ99}. 
Let us equip $M$ with a Hermitian ample holomorphic line bundle  $(L,h)\to M$  with positive curvature $\omega$, 
that is, locally 
\begin{equation}\label{omega}
 \omega = \frac{1}{2i\pi}\partial \bar{\partial}\log \|s\|^2_h>0,
 \end{equation}
where $s$ is any local non vanishing holomorphic section of $L$. Let $g= \omega (\cdot, J\cdot)$ be the associated K\"ahler metric, that is fixed from now on. Consider a rank $r$ holomorphic Hermitian bundle $(E,h_E)$ over $M$.
The space $H^0(M,E\otimes L^{ d})$ of holomorphic sections of $E\otimes L^{d}:=E\otimes L^{\otimes d}$ is non trivial for $d$ large enough, 
and can be equipped with the $\mathscr{L}^2$ Hermitian product 
\begin{equation}\label{prod}
(s,t)\in (H^0(M,E\otimes L^d))^2\mapsto \langle s,t\rangle = \int_M \langle s(x),t(x)\rangle_{h_d}\frac{\omega^n}{n!},
\end{equation}
where 
$h_d:= h^E\otimes h^d.$
This product induces a Gaussian measure $\mu_d$ over $H^0(M,E\otimes L^d)$, 
defined for any Borelian $U\subset H^0(M,E\otimes L^d)$ by
\begin{equation}\label{mesure}
\mu_d (U)=\int_{s\in U} e^{-\frac12 \|s\|^2} \frac{\mathrm{d} s}{(2\pi)^{N_d}},
\end{equation}
where $N_d $ denotes the complex dimension of $H^0(M,E\otimes L^d)$, and $\mathrm{d} s$ denotes the Lebesgue measure associated with the Hermitian product~(\ref{prod}).
\begin{remark} Given the previous $\mathscr{L}^2$ Hermitian product,  an equivalent way to define a random section $s$ with respect to the previous metric $\mu_d$ is the following:
if $(S_i)_{i\in \{1, \dots, N_d\}}$
denotes a unitary basis of this space, then 
$$ s =\sum_{i=1}^{N_d} a_i S_i$$
follows the law $\mu_d$ if the random variables $\sqrt 2 a_i$ are i.i.d standard complex Gaussians, 
that is $\Re a_i$ and $\Im a_i$ are independent, centered Gaussian variables with variance equal to $1/2$.
Note  also that for any event depending only on the vanishing locus $Z(s)$ of $s\in H^0(M,E\otimes L^d)$, the probability measure $\mu_d$ can be replaced by the invariant measure over the unit sphere $$ \mathbb S H^0(M,E\otimes L^d)$$ for the product~(\ref{prod}), or equivalently the Fubini--Study measure on the linear system $\mathbb P H^0(M,E\otimes L^d)$.
\end{remark}

\begin{example}[Random polynomials]\label{example} When $M$ is the projective space $\C\mathbb{P}^n$, the line bundle $(L, h)$ is the degree
	$1$ holomorphic line bundle $(\mathcal O(1),h_{\mathrm{FS}})$  equipped with the standard Fubini--Study metric and $(E,h^E) = (\C\mathbb{P}^n\times \C^r,h_0)$ is the trivial rank $r$ bundle equipped with the standard Hermitian product, then the space  of global section $H^0(M,E\otimes L^d)$ is naturally identified with to the space of $r$-uples of  degree $d$ homogeneous
	polynomials in $n + 1$ variables
	$ (\C_d^{hom}[Z_0 , \cdots , Z_n ])^r$. In this case, if $(e_i)_{i\in \{1, \cdots, r\}}$ denotes the
	standard basis of $\C^r$, the family 
	$$\left(\sqrt{ \frac{(n+d)!}{n!i_0!\cdots i_n!})}
	Z_0^{i_0}\cdots Z_n^{i_n}\otimes e_i\right)_{i_0+\cdots +i_n =d, \ i\in \{1, \cdots, r\}}$$ 
	is a unitary basis for the Hermitian metric~(\ref{prod}).
\end{example}
\subsection{Statement of the main results} Our main result estimates the proportion of a complex submanifold $Z$ where its bisectional, holomorphic sectional, Ricci, and scalar curvature of $Z$ has negativity properties. To do this, 
for any $a\in \R$ and any complex submanifold $Z$ of $M$,  we define
\begin{eqnarray*}
\vol(\hbc_Z<a)&:=&\vol \left\lbrace x\in Z, \ \sup_{X,Y\in  T_xZ} \hbc_{Z,g_{|Z}} (X,Y)<a\right\rbrace
\end{eqnarray*}
the volume of the part of $Z$ where the holomorphic bisectional curvature is smaller than $a$ and by
\begin{equation}\label{kappa0}
\density(\hbc_Z<a):= \frac{\vol(\hbc_Z<a)}{\vol Z}
\end{equation}
the volume fraction of $M$ where the holomorphic bisectional curvature is smaller than $a$. 
We emphasize that the curvature of $Z$ is computed for the restriction $g_{|Z}$ to $Z$.
Similarly, we define $\density(\hc_Z<a)$, 
$\density(\ricci_Z<a)$ and $\density(\scal_Z<a)$.
The main theorem of the paper computes the average of $\density(\hbc_{Z(s)}<a)$, where $Z(s)$ is the zero locus of $s\in H^0(M,E\otimes L^d)$. 
\begin{theorem}\label{theorem1}
	Let $n\geq 2$ and $1\leq r\le n-1$ be integers, $M$ be a complex projective manifold of dimension $n$, $(E,h^E)\to M$ be a rank $r$ holomorphic vector bundle equipped with a Hermitian metric, $(L,h)\to M$ be an ample holomorphic line bundle equipped with a Hermitian metric with positive curvature $\omega$, and $g$ be the associated K\"ahler metric. 
	Then there exists $C>0$ such that for any sequence $(a_d)_{d\in \Nn}$ of positive reals the following holds:
	\begin{enumerate}
		\item (holomorphic bisectional curvature) if 
		$3r\geq 2n-1$, 
	$$\forall d\gg1, \ 
	\mathbb E_{\mu_d}	[
	\density(\hbc_{Z(s)}<-a_d)
]
\geq 1-  C\left(\frac{a_d+C}{d}\right)^{3r-2n+2}.$$
	\item (holomorphic sectional curvature) If $2r\geq n$,
			$$
			\forall d\gg1, \ 
			\mathbb E 
		[
		\density(\hc_{Z(s)} <-a_d)
		]\geq 
		1- C\left(\frac{a_d+C}{d}\right)^{2r-n+1}.$$
		\item (Ricci curvature) $\forall d\gg1, $  $\displaystyle\mathbb E 
		[
		\density(\ricci_{Z(s)} <-a_d)
		]\geq 1-C\left(\frac{a_d+C}{d}\right)^{r(n-r)-(n-r-1)}.$
\item (Scalar curvature) 
$\displaystyle \forall d\gg1, \ 
	\mathbb E 
[
\density(\scal_{Z(s)} <-a_d)]\geq 1-
C \left(\frac{a_d+C}{d}\right)^{\frac{1}2r(n-r)(n-r+1)}.$
	\end{enumerate}
\end{theorem}
The condition on the codimension that appears in the case of holomorphic sectional and bisectional curvature in the theorem is optimal. In fact, if these relations are not satisfied, then, for every point of the submanifold, there will always be directions in which the holomorphic sectional and bisectional curvature are equal to that of the ambient manifold, which can be positive, see section~\ref{sysy}.

Remark that for any $d\in \Nn^*$, 
$$ \inf_{P\in \C^{hom}_d[Z_0,\cdots, Z_n]} \density(\scal_{Z(P)}<0)= 0,$$
and  similarly for the other curvatures. This can be seen by the smoothing of $d$ generic hyperplanes: by making increasingly smaller perturbations of this singular hypersurface, one can obtain smooth hypersurfaces with a volume fraction arbitrarily close to $1$ of the region the curvature is positive. An analogous construction can be made in higher codimension. Moreover, note that \emph{any} complex curve in $\C\mathbb{P}^2$ has points with positive curvature. These are the inflexion points, and generically there are $3d(d-2)$ of them. Similarily, let $T=\C^n/\Lambda$ be a complex torus equipped with the standard flat metric. Then, there is no hypersurface with negative holomorphic sectional curvature, hence with negative holomorphic bisectional curvature. Indeed, the covering of the hypersurface would be a complete hypersurface in $\C^n$ with negative sectional curvature, which is not possible~\cite{yang1977curvatures}. 

For $n=2$ and $r=1$,  then $Z(s)$ is a complex curve and the holomorphic bisectional curvature is the Gauss curvature. In this case, a weaker version Theorem~\ref{theorem1} has been proved in~\cite{AG2}.

\begin{corollary}[Concentration in probability]\label{coro2}
	Under the hypotheses of Theorem~\ref{theorem1},
for any $\epsilon \in (0,1]$,
	\begin{enumerate}
		\item (holomorphic bisectional curvature) if 
		$3r\geq 2n-1$ and  $0<\eta < \epsilon (3r-2n+2),$
		$$
\mu_d	\left[
		\density(\hbc_{Z(s)}<-d^{1-\epsilon})>1-d^{-\eta}
		\right]\to_{d\to \infty} 1.$$
		\item (holomorphic sectional curvature) If $2r\geq n$ and $0<\eta < \epsilon(2r-n+1)$,
		$$
\mu_d	\left[
\density(\hc_{Z(s)}<-d^{1-\epsilon})>1-d^{-\eta}
\right]\to_{d\to \infty} 1.$$

		\item (Ricci curvature) If  $0<\eta < \epsilon(r(n-r)-(n-r-1))$
$$		\mu_d	\left[
		\density(\ricci_{Z(s)}<-d^{1-\epsilon})>1-d^{-\eta}
		\right]\to_{d\to \infty} 1.$$
		\item (Scalar curvature) 
		If  $0<\eta < \epsilon(\frac{1}2r(n-r)(n-r+1))$
		$$		\mu_d	\left[
		\density(\scal_{Z(s)}<-d^{1-\epsilon})>1-d^{-\eta}
		\right]\to_{d\to \infty} 1.$$
	\end{enumerate}
\end{corollary}
Note that, in a slightly different direction, the authors proved in in~\cite[Theorem 1.6]{AG} that $$\mu_d
\left[s\in H^0(E\otimes L^d), \min_{x\in Z(s)}K_{min}(x) > -d^{3(n+1)}
\right]\to 1,$$
where $K_{min}(x)$ is the infimum of the sectional curvature of $Z(s)$ at $x$.
This implies the same for the holomorphic bisectional curvature, since the latter can be expressed as the sum of two sectional curvatures. 

Before stating the next corollary,
let us consider the probability space $(\mathcal H, \mu)$,
where 
$$ \mathcal H := \prod_{d=1}^\infty H^0(M,E\otimes L^d)$$
and $\mu$ is the product of the measures $\mu_d$. 
\begin{corollary}[Almost-sure convergence]\label{almost} Under the hypotheses of Theorem~\ref{theorem1}, 
\begin{enumerate}
		\item (holomorphic bisectional curvature) if 
	$3r-2n+2>1$, then, for almost all sequence $(s_d)_{d\geq 1}$ of $\mathcal H$, 
	$\density (\hbc_{Z(s_d)}<-1)\to_{d\to \infty} 1.$
	\item (holomorphic sectional curvature) If $2r-n+1>1$, 
	then, for almost all sequence $(s_d)_{d\geq 1}$ of $\mathcal H$, 
	$\density (\hc_{Z(s_d)}<-1)\to_{d\to \infty} 1.$
	\item (Ricci curvature) If $r(n-r)-(n-r-1)>1,$
	then, for almost all sequence $(s_d)_{d\geq 1}$ of $\mathcal H$, 
	$\density (\ricci_{Z(s_d)}<-1)\to_{d\to \infty} 1.$
	\item (scalar curvature) If $\frac{1}2r(n-r)(n-r+1)>1$, then, for almost all sequence $(s_d)_{d\geq 1}$ of $\mathcal H$, 
	$\density (\scal_{Z(s_d)}<-1)\to_{d\to \infty} 1.$
\end{enumerate}
\end{corollary}

Finally, let us state a deterministic consequence of Theorem~\ref{theorem1}:
\begin{corollary}[Existence]\label{coro}
		Under the hypotheses of Theorem~\ref{theorem1},
	for any $\epsilon \in (0,1]$,
	\begin{enumerate}
		\item (holomorphic bisectional curvature) if 
		$3r\geq 2n-1$ and  $0<\eta < \epsilon (3r-2n+2),$
			there exists a sequence of sections $(s_d)_{d\geq d_0}$ with $s_d\in H^0(M,E\otimes L^d)$ such that $Z(s_d)\subset M$ is a smooth $(n-r)$-submanifold and
		$$
			\density(\hbc_{Z(s)}<-d^{1-\epsilon})>1-d^{-\eta}.$$
		\item (holomorphic sectional curvature) If $2r\geq n$ and $0<\eta < \epsilon(2r-n+1)$,
			there exists a sequence of sections $(s_d)_{d\geq d_0}$ with $s_d\in H^0(M,E\otimes L^d)$ such that $Z(s_d)\subset M$ is a smooth $(n-r)$-submanifold and
		$$
		\density(\hc_{Z(s)}<-d^{1-\epsilon})>1-d^{-\eta}.$$
		\item (Ricci curvature) If  $0<\eta < \epsilon(r(n-r)-(n-r-1))$, 	there exists a sequence of sections $(s_d)_{d\geq d_0}$ with $s_d\in H^0(M,E\otimes L^d)$ such that $Z(s_d)\subset M$ is a smooth $(n-r)$-submanifold and
		$$		
		\density(\ricci_{Z(s)}<-d^{1-\epsilon})>1-d^{-\eta}
		.$$
		\item (Scalar curvature) 
		If  $0<\eta < \epsilon(\frac{1}2r(n-r)(n-r+1))$, 	there exists a sequence of sections $(s_d)_{d\geq d_0}$ with $s_d\in H^0(M,E\otimes L^d)$ such that $Z(s_d)\subset M$ is a smooth $(n-r)$-submanifold and
		$$	
		\density(\scal_{Z(s)}<-d^{1-\epsilon})>1-d^{-\eta}
		.$$
	\end{enumerate}
%
\end{corollary}
Recently J.-P. Mohsen~\cite{mohsen} proved that  
\begin{enumerate}
\item if $4r\geq 3n-1$ then, for every sufficiently large $d$, 
	there exists a  complete intersections $Z_d$ of degree $d$ and dimension $n-r$ such that $g_{|Z_d}$ has
	negative 	holomorphic bisectional curvature in $M$.
	\item if  $3r\geq 2n$, the same holds for the holomorphic sectional curvature. 
	\item if $r\geq 2$, the same holds for the Ricci curvature. 
	\item For $n\geq 3$, the same holds for the scalar curvature. 
	\end{enumerate}
	For the proof of his theorem, Mohsen used Donaldson's method~\cite{donaldson}, which is a subtle construction of holomorphic sections with a prescribed lower positive bound of the norm of their derivatives. Our proof does not use Donaldson's construction at all. In fact, from a probabilistic point of view, Donaldson's sections are exponentially rare. Moreover, our dimensional conditions are milde. On the other side, we cannot recover his result.
	
\noindent
{\bf Acknowledgments.} The research leading to these results has received funding from the French Agence nationale de la ANR-20-CE40-0017 (Adyct).

\section{Symmetric complex bilinear maps and distance to the discriminant}\label{sysy}
The proof of Theorem~\ref{theorem1} involves the second derivative of a section $s\in H^0(X,E\otimes L^d)$ at a point $x\in Z(s)$. For large degree $d$, this derivative converges to a symmetric bilinear complex map from $T_xZ(s)$ to $E\otimes L^d$, see Proposition \ref{pasteque}. In this paragraph, we develop some elementary results for the space of such maps.

For any $r\in\{1,\dots, n\}$, define   $ \sym_\C(n-r,r)$ to be the set of  complex bilinear symmetric maps from $\C^{n-r}$ with values in $\C^r$, that is the set of maps of $\C^{n-r}\times \C^{n-r}$ in $\C^r$ which are complex linear in both variables and symmetric.
Let 
\begin{eqnarray*} S_1&:=& \{
T\in \sym_\C(n-r,r), \ \exists(x,y)\in (\C^{n-r} \setminus\{0\})^2, \ T (x,y)=0
\}\\
S_2 &:=& \{
T\in \sym_\C(n-r,r), \ \exists x\in \C^{n-r} \setminus\{0\}, \ T (x,x)=0
\}.\\
S_3& := & \{
T\in \sym_\C(n-r,r), \ \exists x\in \C^{n-r} \setminus\{0\}, \ T (x,\cdot)=0\in L(\C^{n-r},\C^r)
\}.
\end{eqnarray*}
\begin{remark}
In the proof of Theorem~\ref{theorem1}, $S_1$ is related to the holomorphic bisectional curvature, $S_2$ to the holomorphic sectional curvature, and $S_3$ to the Ricci curvature. The scalar case does not need these preliminaries.
\end{remark}
Let us also define
\begin{eqnarray*}
H_1&:= & H^0\left((\C \mathbb{P}^{n-r-1})^2,
(\mathcal O(1)\boxtimes \mathcal O(1))^r\right)\\
H_2 &:= & H^0\left(\C \mathbb{P}^{n-r-1},
(\mathcal O(2))^r\right)\\
H_3&:= & H^0\left(\C \mathbb{P}^{n-r-1},
(\mathcal O(1))^{r(n-r)}\right),
\end{eqnarray*}
where  $\mathcal O(d)$ denotes the degree $d$ line bundle over $\C \mathbb{P}^{n-r-1}$.
For any $k\in \{1, 2, 3\}, $ denote
$$
\Phi_k :  \sym_\C(n-r,r)\to H_k$$
the following natural isomorphisms: 
\begin{eqnarray*}\forall T\in \sym_\C(n-r,r), \ 
	\Phi_1(T) &=& (T_i(X,Y))_{i\in \{1,\cdots, r\}},\\
	\Phi_2(T) &=& (T_i(X,X))_{i\in \{1,\cdots, r\}},\\
	\Phi_3(T) &=& (T_i(e_j,X))_{i\in \{1,\cdots, r\}, j\in \{1, \cdots, n-r\}},
\end{eqnarray*}
where $T_i$, $i\in \{1,\cdots, r\}$, denotes the 
$i$-th component of $T$, where $(e_j)_{j\in \{1, \cdots, r\}}$ denotes the standard basis of $\C^{n-r}$ and where we identified 
the holomorphic sections of $\mathcal O(1)$ (resp. $\mathcal O(2))$ with the linear forms (resp. quadratic polynomials) in $X\in \C^{n-r}$. 

For any $i\in \{1, 2, 3\}, $ let 
$$
	\Delta_i \subset H_i$$ 
	 be the discriminant loci of $H_i$, that is the set of sections of $H_i$ which do not vanish transversally. 
	 Finally, let
$$
	\Delta_{S_i} := \Phi_i^{-1}(\Delta_i).$$
	\begin{lemma}\label{codim}
	Using the above notations, for any $r\in\{1,\dots,n\}$ the following holds.	
			
		\begin{enumerate}
			\item If $3r>2n-2$, then
			$\codim_\C \Delta_1 \geq r-2(n-r-1)=3r-2n+2$.
			\item If $2r>n-1$, then
$\codim_\C \Delta_2 \geq r-(n-r-1)=2r-n+1$.
			\item If $\codim_\C \Delta_3 = r(n-r)-(n-r-1)$.
		\end{enumerate}
	\end{lemma}
\begin{proof}
Let $F$ be a rank $e$ holomorphic vector bundle over a projective manifold $M$ of dimension $m$. Suppose that $e>m$. Thus, the discriminant $\Delta$ in $H^0(M,F)$ is the subset of sections that vanish somewhere in $M$.
Define the incidence variety to be $$\Sigma=\{(s,x)\in H^0(M,F)\times M, s(x)=0\}.$$ By construction, $\Delta=\pi(\Sigma)$, where $\pi:H^0(M,F)\times M\rightarrow H^0(M,F)$ denotes the first projection.
Suppose that $\Sigma$ has codimension $e$ in $H^0(M,F)\times M$. Then, by~\cite[Theorem 10]{mohsen}, the codimension of $\Delta$ in $H^0(M,F)$ is equal or greater than $e-m$.

Thus, the result follows by taking 
\begin{enumerate}
			\item $M=(\C \mathbb{P}^{n-r-1})^2$ and $F=(\mathcal O(1)\boxtimes \mathcal O(1))^r$ for $\Delta_1$, 
			\item $M=(\C \mathbb{P}^{n-r-1})^2$ and $F=(\mathcal O(2))^r$ for $\Delta_2$, 
			\item $M=(\C \mathbb{P}^{n-r-1})^2$ and $F=(\mathcal O(1))^{r(n-r)}$ for $\Delta_3$.  
		\end{enumerate}	
			In all three cases, using the fact that $\mathcal O(1)$ and $\mathcal O(2)$ are very ample, it is easy to show that the map $(s,x)\in H^0(M,F)\times M\mapsto s(x)\in F$ is a surjection, and thus, the respective incidence variety has the correct codimension, equal to the rank of $F$.
				\end{proof}
	\begin{lemma}\label{comparaison}
	Using the above notations, for any $T\in \sym_\C(n-r,r)$ the following holds.
	\begin{enumerate}
\item 	$\displaystyle \min_{ (x,y)\in (\mathbb S^{2n-2r-1})^2 } |T(x,y)|_{\C^r}
	= \min_{([X],[Y])\in (\C \mathbb{P}^{n-r-1})^2 }\|\Phi_1(T)([X],[Y])\|_{\mathrm{FS}}$,
	\item 	$\displaystyle \min_{x\in \mathbb S^{2n-2r-1}} |T(x,x)|_{\C^r}
	= \min_{[X]\in \C \mathbb{P}^{n-r-1} }\|\Phi_2(T)([X])\|_{\mathrm{FS}}$,
	\item 	$\displaystyle \min_{x\in \mathbb S^{2n-2r-1}} \|T(x,\cdot)\|_{\C^r}
	= \min_{[X]\in \C \mathbb{P}^{n-r-1}}\|\Phi_3(T)([X])\|_{\mathrm{FS}}$.
	\end{enumerate}
		\end{lemma}
	\begin{proof}The proof is straightforward, since for any homogeneous degree $d$ polynomial $P$,  $$\forall [X]\in \C \mathbb{P}^{n-r-1}, \ \|P([X])\|_{\mathrm{FS}}= \frac{|P(X)|}{\|X\|^d}$$
	by definition of the Fubini--Study norm.
		\end{proof}
\begin{proposition}\label{lembi} Using the above notations, the following holds. 
	\begin{enumerate}
		\item\label{bibi} 
	$3r> 2n-2\Leftrightarrow 
{S_1}\subset	\Delta_{S_1} .$
	\item\label{hoho} 
	$2r> n-1\Leftrightarrow 
	{S_2} \subset \Delta_{S_2}.$
	\item\label{riri} 
	$	{S_3} \subset \Delta_{S_3}.$
	\end{enumerate}
\end{proposition}
\begin{proof} 
	For $i\in \{1, 2,3\}$, notice that
		\begin{equation}\label{equi}\forall T\in \sym_\C(n-r,r),\  T\in {S_i} \Leftrightarrow \Phi_i(T)^{-1}(0)\neq \emptyset.
			\end{equation}
	 Since $\mathcal O(1)\boxtimes \mathcal O(1)$ is a very ample holomorphic line bundle over $(\C \mathbb{P}^{n-r-1})^2$, by Bertini Theorem, 
	\begin{itemize}
		\item if $2(n-r-1)<r$,  
		for any $s\in H_1\setminus \Delta_1,$ $s^{-1}(0)$ is empty;
		\item if $2(n-r-1)\geq r$,  
		for any $s\in H_1\setminus  \Delta_1,$ $s^{-1}(0)$ is a $r$-codimension smooth submanifold of $(\C \mathbb{P}^{n-r-1})^2$. In particular, it is not empty.
	\end{itemize}
	The latter and \eqref{equi} proves~\eqref{bibi}.
Let us prove~(\ref{hoho}). 
 Since $\mathcal O(2)$ is a very ample holomorphic line bundle over $\C \mathbb{P}^{n-r-1}$, by Bertini Theorem,
	\begin{itemize}
		\item if $n-r-1<r$,  
		for any $s\in H_2\setminus \Delta_2,$ $s^{-1}(0)$ is empty;
		\item if $n-r-1\geq r$,  
		for any $s\in H_2\setminus \Delta_2,$ $s^{-1}(0)$ is a $r$-codimension smooth submanifold of $\C \mathbb{P}^{n-r-1}$. In particular, it is not empty.
	\end{itemize}
		The latter and (\ref{equi}) proves~\eqref{hoho}.
Finally, let us prove~\eqref{riri}.
Since $\mathcal O(1)$ is a very ample holomorphic line bundle over $\C \mathbb{P}^{n-r-1}$, by Bertini Theorem,
	\begin{itemize}
		\item if $n-r-1<r(n-r)$,  
		for any $s\in H_3\setminus \Delta_3,$ $s^{-1}(0)$ is empty;
		\item if $n-r-1\geq r(n-r)$,  
		for any $s\in H_3\setminus \Delta_3,$ $s^{-1}(0)$ is a $r(n-r)$-codimension smooth submanifold of $\C \mathbb{P}^{n-r-1}$. In particular, it is not empty.
	\end{itemize}
	Since only the first case can occur,
			the latter and (\ref{equi}) prove
 \eqref{riri}.
\end{proof}

In the following Lemma, we use the notation above.
\begin{lemma}~\label{dist}Using the above notations,  the following holds. 
	\begin{enumerate}
\item Assume $3r>2n-2$. Then,
for any $s \in H_1 $, 
	$$\dist (s, \Delta_1) = (n-r)^{-1}\min_{(X,Y)\in (\C \mathbb{P}^{n-r-1})^2}\|s(X,Y)\|_{\mathrm{FS}}.$$
\item Assume $3r>2n-2$. Then,
for any $s \in H_2 $, 
$$\dist (s, \Delta_2) = (n-r)^{-1}\min_{X\in \C \mathbb{P}^{n-r-1}}\|s(X)\|_{\mathrm{FS}}.$$
\item For any $s \in H_3 $, 
$$\dist (s, \Delta_3) = (n-r)^{-1}\min_{X\in \C \mathbb{P}^{n-r-1}}\|s(X)\|_{\mathrm{FS}}.$$
\end{enumerate}
In all three cases, the distance is computed with respect to the $\mathscr{L}^2$-metric induced by the Fubini--Study metric on $\mathcal{O}(1)$, see Equation \eqref{prod} and Example \ref{example}.
\end{lemma}
\begin{proof}
	Let us give the details of proof for the holomorphic bisectional curvature. The formula for the other curvatures is proved an analogue way.
We follow~\cite[Lemma 3.8]{ancona-complete}, see also~\cite{raffalli}.  

By the hypothesis on $(n,r)$, any section of $H_1 = H^0\left((\C \mathbb{P}^{n-r-1})^2,
(\mathcal O(1)\boxtimes \mathcal O(1))^r\right)$ cannot vanish transversely. In other words, the discriminant locus is the subvariety of sections with non-trivial vanishing locus. 
For any $(x,y) \in (\C \mathbb{P}^{n-r-1})^2$, let $\Delta_{x,y}$ be the space of sections of $H_1$
	that vanish at $(x,y)$, so that 
	$$ \Delta = \bigcup_{(x,y)\in (\C \mathbb{P}^{n-r-1})^2} \Delta_{x,y}$$
and for any $s\in H_1$, 
$$\dist(s,\Delta)= \min_{(x,y)\in (\C \mathbb{P}^{n-r-1})^2}
\min_{s'\in \Delta_{x,y}} \|s-s'\|_{\mathscr L^2}.$$
	Since the metric on $H_1$ is invariant under the isometries of $\C \mathbb{P}^{n-r-1}\times \C \mathbb{P}^{n-r-1}$, one can assume that $x=[1:0\cdots : 0] $ and $y=[1:0\cdots : 0]$. Besides, since the standard monomials are orthogonal in $H_1$ (see Example \ref{example}), we have
	\begin{eqnarray*} \Delta_{x,y} 
	= \left(\Delta_x \otimes \C_1^{\mathrm{hom}}[Y_0,\dots,Y_{n-r-1}]
	\oplus \C_1^{\mathrm{hom}}[X_0,\dots,X_{n-r-1}]\otimes 
	\Delta_y
	\right)^r,
\end{eqnarray*}
where 
$$ \Delta_x = \vect(X^\alpha)_{\alpha = (0, \alpha_1, \cdots, \alpha_{n-r-1})}
	^{|\alpha|=1 } \text{ and }
	\Delta_y = \vect(Y^\alpha)_{\alpha = (0, \alpha_1, \cdots, \alpha_{n-r-1})}
	^{|\alpha|=1 }.
	$$

We then write $s$ in the orthonormal decomposition (see Example~\ref{example}) $$s = (n-r) X_0Y_0\otimes \sum_{i=1}^r a_i  e_i+\tau \in \Delta_{x,y}^\perp\oplus \Delta_{x,y}.$$
Then, for any $s'\in \Delta_{x,y}$, we have
\begin{equation}\label{minimo}
\|s-s'\|^2_{\mathscr L^2}=
\sum_{i=1}^r|a_i|^2 +\|\tau-s'\|^2_{\mathscr L^2}
\end{equation}
so that the minimum of \eqref{minimo} is reached for $s'=\tau$, that is,
$$\min_{s'\in \Delta_{x,y}} \|s-s'\|_{\mathscr L^2}=\sum_{i=1}^r|a_i|^2.$$
	Finally,
	$$ \sum_{i=1}^r|a_i|^2= \|s(x,y)\|^2_{\mathrm{FS}}(n-r)^{-2},$$
	hence the result.
	\end{proof}

\begin{proposition}\label{burg} Using the above notations, the following holds. There exists $ C>0$ such that for any $\epsilon>0$ small enough,
	$$ \forall i\in \{1,2,3\}, \ \mu \left\lbrace
	s\in   H_i, \ \dist(s,\Delta_i)
	\leq \epsilon\|s\|_{\mathscr{L}^2} \right\rbrace\leq C\epsilon^{2\codim \Delta_i}.$$
Here, $\mu$ is the Gaussian probability measure on $H_i$ induced by the Fubini--Study metric on $\mathcal{O}(1)$, see Example \ref{example}. 
 \end{proposition}
\begin{proof} The estimate in the statement is equivalent to the existence of a constant $C'>0$ such that
\begin{equation}\label{sfera} \forall i\in \{1,2,3\}, \ \nu \left\lbrace
	s\in  \mathbb S H_i, \ \dist(s,\mathbb S\Delta_i)
	\leq \epsilon \right\rbrace\leq C'\epsilon^{2\codim \Delta_i}
	\end{equation}
	for any $\epsilon$ small enough, where $\mathbb S H_i$ denotes the unit sphere in $H_i$, $S\Delta_i:=\Delta_i\cap \mathbb S H_i$,  and $ \nu$ is the uniform probability measure on $\mathbb S H_i$.
Now, the discriminant $\mathbb S\Delta_i$ is a complex algebraic subset of $\mathbb S H_i$ and its codimension in $\mathbb S H_i$ equals the codimension of $\Delta_i$ in $H_i$. Thus,  by applying~\cite[Theorem 1.3]{basu2023hausdorff}, one obtains \eqref{sfera}. Hence, the result follows.
	\end{proof}

\section{A formula for the holomorphic bisectional curvature }
The main result of this section is Proposition \ref{proposition2}, which provides a useful formula for the holomorphic bisectional curvature of a complex submanifold $Z(s)$ in terms of $s$. 

We begin by recalling some classical Riemannian facts. 
Let $Z$ be a submanifold of the Riemannian manifold $(M,g)$, $x\in Z$ and let
\begin{eqnarray*}
	\sigma : T_x Z\times T_x Z &\to& N_x Z\\
	(X,Y)& \mapsto & (\nabla_X Y)^\perp
\end{eqnarray*}
where $NZ\subset TM$ denotes the normal bundle of $Z$, and $\nabla$  the Levi--Civita connection associated to $g$.
The Gauss--Codazzi's equations~\cite[Theorem 3.6.2]{jost2008riemannian} allow us to compute the Riemann curvature $R^Z$ of $(Z,g_{|Z})$ in terms of $R^M$ and $\sigma$:
\begin{eqnarray}\label{gauss}
	\ \forall X,Y,V,W\in T_xZ,\ 
	\langle R^Z(X,Y)V,W\rangle&=& \langle R^M(X,Y)V,W\rangle  +
	\\ \nonumber &&\langle \sigma(Y,V),\sigma(X,W)\rangle- \langle \sigma(X,V),\sigma(Y,W)\rangle.
\end{eqnarray}
Let $E\to M$ be a smooth real vector bundle equipped with a metric $h$ and a metric connection $\nabla$.
Recall that the second covariant derivative $\nabla^2$ is defined by
\begin{equation}\label{covariant}
\forall x\in M, \ \forall V,W\in T_xM, \ \nabla^2_{VW} = \nabla_V\nabla_W-\nabla_{\nabla^{\mathrm{LC}}_V W}.
\end{equation}
The following proposition computes the curvature of the zero locus of a transverse section of $E$ in terms of its derivatives. 
\begin{proposition}[\cite{AG}]\label{sig}Let $n\geq 1$ and $1\leq r\leq n$ be integers, $(M,g)$ be a smooth Riemannian manifold of dimension $n$, $E\to M$ be a rank $r$ smooth real vector bundle equipped with a metric $h$ and a metric connection $\nabla$. Let $s\in \mathscr{C}^\infty(M,E)$ be a smooth section vanishing transversally on $Z(s)\subset M$. Then,
	for any $x\in Z(s)$,	
	\begin{equation}\label{uranus}
	\langle \sigma(V,W), \sigma(X,Y)\rangle_g  =
	( \nabla s G \nabla s^*)^{-1}(\nabla^2_{V,W}s) (\nabla^2_{X,Y} s)
	\end{equation}
	where $G:T^*M\rightarrow TM$ is the isomorphism given by $g(G(\alpha),\cdot)=\alpha$, for any $\alpha\in T^*M$.
	\end{proposition}

The following proposition computes the holomorphic bisectional  curvature of a complex submanifold of a K\"ahler manifold, in terms of a (not necessarily holomorphic) section $s$  of $E$ that vanishes along it. 
\begin{proposition}\label{proposition2}
Let $(M,g,J)$ be a K\"ahler manifold, $E\to M$ be a complex vector bundle of rank $r$, and $s\in \mathscr{C}^\infty(M,E)$ vanishing transversely along a $J$-complex submanifold $Z(s)$. Then,  
for any $x\in Z(s)$ and any pair of unit vectors $X,Y\in T_xZ(s)$,
\begin{eqnarray*}
	\hbc_{Z(s)} (X,Y)&=&  \hbc_{M} (X,Y)
	-2
	(\nabla^E s G (\nabla^E s)^*)^{-1}((\nabla^{E})^2_{X,Y}s)( (\nabla^{E})^2_{X,Y}s),
\end{eqnarray*}
where everything is computed at $x$, and where $G:T^*M\rightarrow TM$ is the isomorphism given by $g(G(\alpha),\cdot)=\alpha$, for any $ \alpha\in T^*M$.
\end{proposition}
\begin{proof}
	Since $(M,J,g)$ is K\"ahler, $\nabla J=0$, so that for any $i\in \{1, \cdots, k\}$, using the symmetry of $\sigma$ and the fact that the normal subspace $(T_xZ(s))^\perp\subset T_xM$ is invariant under $J$,
for all $i\in \{1,\cdots, n-r\},$ for all $X\in T_xZ(s),$
\begin{equation}\label{junk}
 (\nabla^{\mathrm{LC}}_{Je_i}(Je_i))^\perp = -(\nabla^{\mathrm{LC}}_{e_i} e_i)^\perp \text{ and }
		(\nabla^{\mathrm{LC}}_{Je_i}(X))^\perp = J(\nabla^{\mathrm{LC}}_{e_i} X )^\perp, 
\end{equation}
where $\perp$ denotes the orthogonal projection onto $(T_xZ(s))^\perp$.
		By the Gauss equations, 
$$
\forall X,Y,V,W\in T_xZ(s),\ 
			\langle R^{Z(s)}(X,Y)V,W\rangle_g= \langle R^M(X,Y)V,W\rangle_g  +
			\langle \sigma(Y,V) , \sigma(X,W)\rangle - 	\langle \sigma(X,V), \sigma(Y,W) \rangle,
$$
		where $R^{Z(s)}$ (resp. $R^M$) denotes the Riemannian curvature of $g_{|Z(s)}$ on $Z(s)$ (resp. $g$ on $M$).
			By~(\ref{junk}) and~(\ref{gauss}), this implies
			\begin{eqnarray*}
\hbc_{Z(s)} (X,Y)&=&  \hbc_{M} (X,Y)
				-2
			\langle \sigma(X,Y),\sigma (X,Y)\rangle,
			\end{eqnarray*} 
		where we used that $\sigma$ is symmetric. 
						By Proposition~\ref{sig}, 
		for any tangent vector $X,Y,V,W$ in $T_xZ(s),$
		\begin{eqnarray*}
		\langle \sigma(Y,V),\sigma(X,W)\rangle = 
			( \nabla s G \nabla s^*)^{-1}(\nabla^2_{Y,V}s) (\nabla^2_{X,W} s),
		\end{eqnarray*}		hence the result. 
	\end{proof}

Let $(T,g)$ and $(E,h)$ be finite dimensional vector spaces equipped with metrics $g$ and $h$. Let 
$ G \in L( T^*,T)$ be defined by $$\forall \alpha\in T^*, \langle G\alpha, \cdot \rangle_g = \alpha.$$ 
The following lemma will allow us to bound the norm of the inverse from above in the latter Proposition~\ref{proposition2}.
\begin{lemma}\label{lemmaT}
	Let $E,T$ as above and  $f\in L(T,E)$. Then, 
	\begin{itemize}
		\item $\forall \alpha\in E^*,\ 
		\alpha fGf^*\alpha\geq 0. $
		\item 
		If $f$ is onto, then, $	f  G f^*$ is inversible and 
		$$\forall v\in E, \ 
		(f  G f^*)^{-1}(v)(v)\geq \|f\|^{-2}\|v\|^2.$$
	\end{itemize}
\end{lemma}
\begin{proof}
	%
	For any $\alpha\in E^*$, $w\in T$, $f\in L(T,E)$, 
	$$ \langle Gf^*(\alpha), w\rangle_g = f^*(\alpha)(w)= \alpha(f(w)), $$
	so that for any $\beta \in E^*$, $$\beta(f(Gf^*(\alpha))) = f^*(\beta)(Gf^*(\alpha))= \langle Gf^*(\beta),Gf^*(\alpha)\rangle_g= \langle f^*(\beta),f^*(\alpha)\rangle_{g^*},
	$$
	where the scalar product on $T^*$ is defined by
	$$ \forall v,w\in T^*, \langle v,w\rangle_{g^*}= \langle Gv,Gw\rangle_g.$$
	In particular, 
	\begin{equation}\label{boss}
	0\leq \alpha(f  G f^*(\alpha))= \|f^*\alpha\|^2_{g^*}\leq \|f^*\|^2 \|\alpha\|^2.
	\end{equation}
	This proves the first point.
	Now, if $f$ is onto, then $fGf^*$ is inversible. Let
	$$ Q= \mat(fGf^*,B^*),$$ be the matrix of $fGf^*$ seen as a bilinear form on $E^*$, and where $B^*$ is an orthonormal basis $B^*$ of $E^*$. Let $B$ the orthonormal dual basis $B$ of $E$. Then,
	$$ \forall v\in E, 
	\ (fGf^*)^{-1}(v)(v)= {}^tV Q^{-1} V,$$
	where $V$ denotes the coordinate vector of $v$ in $B$. Hence,
	$$ \forall v\in E, 
	\ |(fGf^*)^{-1}(v)(v)|\geq (\max \spec Q)^{-1} \|v\|^2\geq \|f\|^{-2}\|v\|^2, $$
	where we used~(\ref{boss}) for the last inequality.
\end{proof}

\section{Bergman kernel and the Bargmann--Fock field}
In this section, we recall the relation between the Bergman kernel associated with $E\otimes L^d$ and the Bargmann--Fock field.
\subsection{The Bargmann--Fock field }

Recall that the	Bargmann--Fock field is defined
by 
\begin{equation}\label{bafo}
\forall z\in \C^n, \ f(z)=\sum_{(i_1,\cdots, i_n)\in \Nn^n}a_{i_0, \cdots, i_n}
{\sqrt{\frac{\pi^{i_1+\cdots +i_n}}{i_1!\cdots i_n!}}}{z_1^{i_1}\cdots z_n^{i_n}}e^{-\frac12 \pi\|z\|^2},
\end{equation}
where the $\sqrt 2 a_I$'s are independent normal complex Gaussian random variables. The associated covariant function equals
\begin{equation}\label{mp}
\forall z,w\in \C^n, \ \mathcal P(z,w):= \mathbb E (f(z)\overline{f(w)})= \exp\left(
-\frac{\pi}2 (\|z\|^2 +\|w\|^2 -2\langle z,w\rangle_{\C^n})
\right).
\end{equation}
The a priori superfluous presence of $\pi$ is in fact consistent with the projective situation. Indeed, the affine Bargmann--Fock is 
the universal local limit of the projective model, see Theorem~\ref{Dai}.
To unify the setting of this section with that of the next section and the rest of the paper,
we consider here that $M=\C^n$ and $L=\C^n \times \C$ with its standard Hermitian metric. 
Then $\mathcal P (z,w)\in L_z\otimes L_w^*.$ Let $\nabla_0$ be the metric connection on $L$ defined by
\begin{equation}\label{nabla}
\nabla_0 1 = \frac12 \pi (\bar \partial -\partial) \|z\|^2,
\end{equation}
whereas the dual connection $\nabla_0^*$ on $L^*$
satisfies $
\nabla_0^* 1^* = -\frac12 \pi (\bar \partial -\partial) \|z\|^2,$
where $1^*$ is the dual of $1$. The reason why we fix the connection $\nabla_0$ is that  the curvature of the trivial connection on $L$ is $0$, while, as we will see below, the curvature of $\nabla_0 $ is positive. 

Note that the constant section $1$ is no longer a holomorphic section for the connection $\nabla$, but 
the section $\sigma_0:= \exp(-\frac12 \pi \|z\|^2)$ is. 
The connection $\nabla_0$ is then the Chern connection for the trivial metric and this holomorphic structure. 
This implies that the section $\mathcal P$ is holomorphic in $z$, and antiholomorphic in $w$. 
Moreover, the
curvature of $\nabla_0$ equals
$$ \mathcal R_0 = \bar\partial \partial \log \|\sigma_0\|^2 = \pi \partial \bar \partial \|z\|^2,$$
and the
curvature form equals
$\frac{i}{2\pi}\mathcal R_0=\frac{i}2 \sum_{i=1}^n \mathrm{d}z_i \wedge \overline{\mathrm{d}z_i}
$
which is the standard symplectic form $\omega_0$ over $\R^{2n}$. 
Moreover, for any $\mathscr C^2$ function $f : B(0,\epsilon)\to \C$,
\begin{eqnarray*} \forall v,w\in \C^n, (\nabla_0)^2_{vw}f(0) 
	& = & 
D^2f(v,w)+	\frac{1}2\pi\sum_{i}(v_i\bar{w_i}-w_i\bar{v_i})f\\
	&=& D^2f(v,w)+\frac{1}2\mathcal R_0(v,w)f.
\end{eqnarray*}
Notice that $\mathcal R_0(v,w) = (\nabla_0)^2_{vw}-(\nabla_0)^2_{wv}$.
Finally, a direct computation shows
\beqe 
\nabla_0^2\nabla_0^{*2} \mathcal P (0,0) &=& 
\pi^2 \sum_{i,j,k,\ell=1}^n(\delta_{ik}\delta_{j\ell}+\delta_{i\ell}\delta_{jk}) \mathrm{d}z_i \otimes \mathrm{d}z_j\otimes \overline{\mathrm{d}w_k}\otimes \overline{\mathrm{d}w_\ell},
\eeqe
see for example ~\cite[Lemma 4.2]{bettigayet}.

\begin{remark} Almost surely the  Bargmann--Fock Gaussian field $f$ is a holomorphic section
for the standard complex structure and the connection  defined by~(\ref{nabla}).
\end{remark}
Let $E=\C^n \times \C^r$ endowed with its trivial metric and 
let $f=(f_i)_{i=1, \cdots, r}$ be $r$ 
independent copies of the Bargmann--Fock field.  Then,
$f$ is a random section of $E\otimes L$,
and its covariance function equals $\mathcal P \text{Id}_{\C^r}.$
We use the connection $(\nabla_0)^r$ (the $r$-product of $\nabla_0$) acting on sections of $E\otimes L$. By an abuse of notation, we continue to use $\nabla_0$ for $(\nabla_0)^r$. 

Let 
$\Sigma_{\mathrm{GOE}}$ be the variance matrix defined by:
\beq\label{GOE}
\Sigma_{\mathrm{GOE}} =  \left(\delta_{ik}\delta_{j\ell}+\delta_{i\ell}\delta_{jk}\right)_{\substack{1\leq i\leq j\leq n\\
		1\leq k\leq l\leq n}}\in M_{\frac{n(n+1)}2}(\C).
\eeq

\begin{remark} A complex bilinear form $T\in \sym_\C (\C^n,\C)$ writes
$$ T= \sum_{1\leq i,j\leq n}T_{ij}\mathrm{d}z_i\otimes \mathrm{d}z_j,$$
with $T_{ij}\in \C$ for any $1\leq i,j\leq n$. Now, if $T$ is a random Gaussian complex bilinear form and satisfies $\cov(T)= \Sigma_\mathrm{GOE}$, 
then the $T_{ij}$ are independent if $i\leq j$, 
$ \var (T_{ij})= 1$ is $i<j$ and $\var(T_{ii})=2.$ Equivalently, the Gaussian measure on $\sym_\C (\C^n,\C)$ equals
\begin{equation}\label{meumeu}
 e^{-\frac{1}{4} \tr(TT^*)} \frac{\prod_{1\leq i\leq j\leq n}\mathrm{d}T_{ij}}{(2\pi)^{\frac{n(n+1)}{2}}2^n },
 \end{equation}
where $T$ is identified with the complex symmetric matrix $(T_{ij})_{1\leq i,j\leq n}.$
\end{remark}

\begin{proposition}\cite[Corollary 4.3]{bettigayet}\label{coaff} Let $f:\C^n \to \C^r$ be $r$ independent copies of the Bargmann--Fock field. 
	Then, 
	\begin{equation*}
	\cov (f(0), \nabla_0 f(0),\nabla_0^2f(0)) = \cov_{\mathrm{BF}},
	\end{equation*}
	where 
	\begin{equation}\label{covBF}
	\cov_{\mathrm{BF}}:=	 \begin{pmatrix}
	1 & 0 & 0\\
	0   & \pi \ \id_{\C^n}  & 0\\
	0&0& \pi^2\Sigma_\mathrm{GOE}
	\end{pmatrix} \otimes \ \id_{\C^r}.
	\end{equation}
\end{proposition}

\begin{remark}\label{remarkalmostsure} Proposition \ref{coaff} implies, in particular, that 
almost surely $\nabla_0f(0)$ is a complex linear map, 
and that $\nabla^2_0f(x)\in \sym_\C(\C^n,\C^r)$, that is, $\nabla^2_0f(0)$ is a bilinear complex map with values in $\C^r$.
\end{remark}

\subsection{The Bergman kernel}\label{bergman}

In this paragraph we assume that the setting and hypotheses of Theorem~\ref{theorem1} are satisfied. The covariance function $K_d$ for the Gaussian field generated by the holomorphic sections $s\in H^0(M, E\otimes L^d)$ is defined by
$$ \forall z,w\in M, 
\ K_d (z,w) = \mathbb E \left[s(z)\otimes (s(w))^*\right]\in (E\otimes L^d)_{|z} \otimes (E\otimes L^{d})_{|w}^*,$$
where the averaging is made for the measure $\mu_d$ given by~(\ref{mesure}), where $L^*$ is the (complex) dual of $L$ and 
$$\forall w\in M, \ \forall s,t\in L^d_w, \  s^* (t) = \langle s,t\rangle_{h^d(w)}.$$
The covariance $K_d$ is the \emph{Bergman kernel}, that is the kernel of the orthogonal projector from $\mathscr{L^2}(M, L^d)$ onto $H^0(M, L^d)$. This fact can be seen through the equations
$$\forall z,w\in M, \ K_d(z,w) = \sum_{i=1}^{N_d} S_i (z)\otimes S_i^*(w),$$
where $(S_i)_i$ is an orthonormal basis of $H^0(M, L^d)$ for the Hermitian product~(\ref{prod}).
Recall that the metric $g$ is induced by the curvature form $\omega$  and the complex structure. 
It is now classical  that the Bergman kernel has a universal rescaled (at scale $\frac{1}{\sqrt d}$) limit, the Bargmann--Fock kernel $\mathcal P$, see~(\ref{mp}) above.
Theorem~\ref{Dai} below quantifies this phenomenon. For this, we need to introduce local trivializations and charts. 

Let $x\in M$ and $R>0$ be such that $2R$ is less than the radius of injectivity of $M$ at $x$. Then the exponential map based at $x$ induces a chart near $x$ with values in $B_{T_xM}(0,2R)$. We identify a point in $M$ with its coordinates. The parallel transport provides a trivialization 
$$\varphi_x : B_{T_xM}(0,2R)\times (E\otimes L)_{x}\to  (E\otimes L)_{|B_{T_xM}(0,2R) }$$ 
which induces a trivialization of 
$(E\otimes L^d\boxtimes
(E\otimes L^d)^*)_{|B_{T_xM}(0,2R)^2}$.
Under this trivialization, the Bergman kernel 
$ K_d$ becomes a map from $T_x M^2 $ with values into $\End\left(E\otimes L^d_x\right)$.

\begin{theorem}(\cite[Theorem 4.2.1]{ma2007holomorphic})\label{Dai} Under the hypotheses of Theorem~\ref{theorem1}, let $m\in \Nn$. Then, there exist $C>0$, such that for any $k\in\{0,\cdots, m\}, $ for any $x\in M$, 
	$\forall z,w\in B_{T_xM}(0,1), $ 
	\beqe \left\|	D^k_{(z,w)}\left(
	\frac{1}{d^n}K_d(\frac{z}{\sqrt d},\frac{w}{\sqrt d}) - \mathcal P(z,w)
	\ \Id_{E\otimes L^d_x}
	\right)
	\right\|	
	\leq C d^{-1}.
	\eeqe
\end{theorem}
We used in fact~\cite[Proposition 3.4]{letpuch} which simplifies the original theorem, which is more precise but takes in account the derivatives of the volume form.

In the sequel, $f$ denotes the trivialization of a section $s$ at scale $1/\sqrt d$, that is
\begin{equation}\label{fd}
f_d := \frac{1}{\sqrt d^n}f\left(\frac{\cdot}{\sqrt d}\right).
\end{equation}

\begin{lemma}\label{orangoutan} Under the hypotheses of Theorem~\ref{theorem1}, let $s\in \mathscr{C}^\infty(M,E\otimes L^d)$.  Let $x\in M$ and $f=(f_i)_{i\in \{1, \cdots, 2r\}}$ be the trivialization of $s$ in the setting of Section~\ref{bergman}. 
	Then,
	\begin{equation*}\label{Pamela}
	 \nabla f(x)= \nabla_0 f(x)= Df(x) \text{ and }\nabla^2 f_d(x) = \frac{1}{\sqrt d^n}
	\left(\nabla_0^2 + \frac{1}dR_0^E \right) f,
		\end{equation*}
where $\nabla_0$ is defined by~(\ref{nabla}).
\end{lemma}	
\begin{proof}
		First, recall that~\cite[Lemma 1.2.4]{ma2007holomorphic}
	for any smooth function $f: B(x,\epsilon)\to \R^r$ and any vector $v\in \R^{2n}$,
	\begin{equation}\label{mite}
	\forall y \in B(x,\epsilon),\  \nabla_v f(y)=
	Df(y)(v)+
	\frac{1}2 
	\left((R^E_x+dR^L_x)(y,v) +dO(|y-x|^2|v|)\right)f(y),
	\end{equation}
	where $R^E_x$ (resp. $R^L_x$) denotes the curvature of $(E,h^E)$ (resp.  $(L,h^L)$) at $x$.
	Consequently, for any $v,w\in \R^{2n}$,
	\begin{multline*}
		\nabla_v \nabla_w f(y) =
		\nabla_v \left(
		Df(y)(w) +\frac12(R_0^E(y,w) +dR_0^L(y,w)
		+dO(|y-x|^2))f(y)\right) \\
		 =
		D^2_{v,w}f(y) +
		\frac12((R_x^E(v,w) +dR_x^L(v,w)+dO(|y-x|)
		)f(y)
		+ \frac12(R_x^E(y,w) +dR_x^L(y,w))+O(d|y-x|^2))df(y)(v) \\
		+
		\frac{1}2 \left(
		(R_0^E+dR_x^L)(y,v) +dO(|y-x|^2)\right)
		\left(Df(y)(w) +\frac12(R_x^E(y,w) +dR_x^L(y,w)
		+dO(|y-x|^2))f(y)\right),
	\end{multline*}
	so that at $y=x$, using again~(\ref{mite}),
	\begin{eqnarray*}
		\nabla_v \nabla_w f(0) &=&
		D^2_{v,w}f(0) +
		\frac12(R_x^E(v,w) +dR_x^L(v,w)
		)f(0)).
	\end{eqnarray*}
	Now, from~(\ref{omega}) and since we are in normal coordinates, 
	$$R_x^L = \mathcal R^0$$  and 
	$\nabla_0$ defined by~(\ref{nabla}) satisfies
	$ \nabla_0 = D +\mathcal R_0,$
	we obtain the result.
	\end{proof}
\begin{proposition}~\cite{bettigayet}\label{pasteque}
Under the hypotheses of Theorem~\ref{theorem1}, let $x\in M$. Under the trivializations above,
	in any orthonormal basis of $T_xM$, 
	\beqe \cov\left(f_d, \nabla f_d, \nabla^2f_d\right)_{|x} = \cov_{\mathrm{BF}}+ O(\frac{1}d),
	\eeqe
	where $f_d$ is defined by(\ref{fd}) and 
	$\cov_{\mathrm{BF}}$ is defined by~(\ref{covBF}).
\end{proposition}
\begin{proof}This is a direct consequence 
	of Proposition~\ref{coaff}, Theorem~\ref{Dai} and~(\ref{Pamela}).
\end{proof}

\section{Proof of the main results}
In this section, we prove  Theorem \ref{theorem1} and Corollaires \ref{coro2} and \ref{almost}.
Let us start by computing the volume of a submanifold $Z(s)$, where $s\in H^0(X,E\otimes L^d)$.
	\begin{proposition}\label{wirtinger}
	Under the hypotheses of Theorem~\ref{theorem1}, for any $d$ large enough, there exists a positive constant $\vol_d$, such that for any smooth transverse section $s\in H^0(X,E\otimes L^d)$,
	 one has the equality $ \vol(Z(s))= \vol_d.$
	Moreover, $$ \vol_d
=
	\frac{n!}{(n-r)!}d^r \vol(M) +O(d^{r-1}).
	$$		
\end{proposition}
\begin{proof}
By Wirtinger theorem \cite{Wirtinger}, we have that the volume of $Z(s)$ coincides with its symplectic volume, that is,
$$\vol(Z(s))= \int_{Z(s)}\frac{\omega^{n-r}}{(n-r)!}.$$
	Remark that the dual of the fundamental  class of $Z(s)$ is the Euler class $[c_r(E\otimes L^d)]\in H^{r,r}(X,\C)$ of $E\otimes L^d$, so that 
	$$
	\int_{Z(s)}\frac{\omega^{n-r}}{(n-r)!}= \int_M c_r(E\otimes L^d)\wedge \frac{\omega^{n-r}}{(n-r)!}.
	$$
	Finally, using that $$c_r(E\otimes L^d)= \sum_{i=0}^r
	c_i(E)\wedge (c_1(L))^{d-i} $$
	and $c_0(E)=1$ (see ~\cite[p. 55]{fulton}) we obtain the asymptotics
	$$\int_M c_r(E\otimes L^d)\wedge \frac{\omega^{n-r}}{(n-r)!}\sim_{d\to\infty}
	\frac{n!}{(n-r)!}d^r \vol(M).
	$$ and hence the result.
\end{proof}
We prove now the main theorem. 
\begin{proof}[Proof of Theorem~\ref{theorem1}]
Since the proofs for the four curvatures are similar, we will provide a detailed proof only for the case of holomorphic bisectional curvature, as it is the most intricate.

Let $(a_d)_d$ be a sequence of positive real numbers. Without loss of generality, we can assume that for every sufficiently large $d$, we have $a_d<cd $, where $c>0$ is a constant that depends only on $X,L$ and $E$ (otherwise, the estimate in the statement of the theorem is vacuous).

We will estimate the expected value of the complement of the event we are interested in, that is, $\mathbb E[
	\density(\hbc_{Z(s),g_{|Z(s)}}>-a_d)
	]$.	By the Kac--Rice formula,
	see for instance~\cite[Theorem 1.3]{letendre2016expected}, 
\begin{equation}\label{KR} \mathbb E 
	[
	\vol(\hbc_{Z(s),g_{|Z(s)}}>-a_d)
	]=
	\int_M \mathbb E\left[{\bf 1}_{\mathrm{hb}(a_d)}(x,s)|\det  \nabla s G \nabla s^*)|^{\frac12} \, \big\vert \, s(x)=0
	\right]\rho_{s(x)}(0)\mathrm{d}x,
	\end{equation}
	where $\rho_{s(x)}(0)$ is
	the density of $s(x)$ at $0$
	and for any $a\in \R$,
	\begin{eqnarray*}
		\mathrm{hb}(a_d)&:=& \left\lbrace (x,s)\in M\times H^0(M,E\otimes L^d), \ x\in Z(s), \ \sup_{X,Y\in  T_xZ(s)} \hbc_{Z(s)} (X,Y)>-a_d\right\rbrace.
	\end{eqnarray*}
Let $x\in M$ and let us compute the expectation in the latter integral. For this, we use the trivializations described in Section~\ref{bergman}. 
Let $$ f_d := \frac{1}{\sqrt d^n}f\left(\frac{\cdot}{\sqrt d}\right).$$
Let us define the random Gaussian variables 
$$ (F,S,{T}):= \left(f_d(x),\nabla f_d(x), \nabla^2 f_d(x)\right)\in \C^{r}\times 
L(\R^{2n},\R^{2r})\times L((\R^{2n})^{\otimes 2},\R^{2r} ). $$
By Proposition~\ref{proposition2},
for any pair of unit vectors $X,Y$ of $\ker S = T_xZ(s),$ 
\begin{equation}\label{xy}
\hbc_{Z(s)} (X,Y)=  \hbc_{M} (X,Y)
-2 d
(SS^*)^{-1}(T^2(X,Y))(T^2(X,Y)).
\end{equation}
Let $$ \mathcal{L}(n,r):=L(\R^{2n},\R^{2r})\times L((\R^{2n})^{\otimes 2},\R^{2r})$$
and $\mathrm{HB}(a_d)\subset \mathcal{L}(n,r)$ be defined by
\begin{eqnarray*}
	\mathrm{HB}(a_d)&:=& \left\lbrace (S,T)\in \mathcal{L}(n,r), \ 
	\displaystyle\sup_{\substack{(X,Y)\in (\ker S)^2 \\ \|X\|=\|Y\|=1}}\hbc_{M} (X,Y)
	-2 d
	(SS^*)^{-1}(T^2(X,Y))(T^2(X,Y))>-a_d \right\rbrace.
\end{eqnarray*}
Hence, it is straightforward to check that 
\begin{equation}\label{reno}
\mathbb E\left[{\bf 1}_{\mathrm{hb}(a_d)}(x,s)|\det  \nabla s G \nabla s^*)|^{\frac12} \, \big\vert \, s(x)=0
\right]\rho_{s(x)}(0) =
d^{r}\mathbb E\left[{\bf 1}_{\mathrm{HB}(a_d)}(S,T)|\det  SS^*|^{\frac12} \, \big\vert \, F=0
\right]\rho_{F}(0),
\end{equation}
where $\rho_F(0)$ denotes the density of $F$ for the measure induced by $\mu_d$. 
Let us define
$$ h = \max_M \|\hbc_{M}\|$$
and 
\begin{eqnarray*}
	\mathrm{\widetilde{HB}}(a_d)&:=& \left\lbrace (S,T)\in \mathcal{L}(n,r), \ 
	\displaystyle\min_{\substack{(X,Y)\in (\ker S)^2 \\ \|X\|=\|Y\|=1}}\|T(X,Y)\|^2< \frac{a_d+h}{2d}\|S\|^{2} \right\rbrace.
\end{eqnarray*}
Then, Lemma~\ref{lemmaT} and Equation~(\ref{xy}) imply that
\begin{equation}\label{aba}
 \mathrm{HB}(a_d)\subset \mathrm{\widetilde{HB}}(a_d).
\end{equation}
By Proposition~\ref{pasteque},
\begin{equation}\label{coco}\cov (F,S,T)= \cov_{\mathrm{BF}}+O(\frac{1}d),
\end{equation}
so that
$$ \cov ((S,T)|F=0)= \cov_{\mathrm{BF}}(S,T)+O(\frac{1}d).$$
Note also that $$
\rho_{F}(0)\to_{d\to \infty}\rho^{\mathrm{BF}}_{F}(0).$$
Since the covariance matrix of $((S,T)|F=0)$ is positive, for $d$ large enough, using~(\ref{aba}), 
we obtain

\begin{multline}\label{mino}
	\mathbb E\left[{\bf 1}_{(\mathrm{ {HB}}(a_d))}|\det  SS^*|^{\frac12} \, \big\vert \, F=0
	\right]
	 \leq  \\
	\frac{1}{C(n,r)}\int_{(S,T)\in L(\C^{n},\C^{r})\times \sym_\C(n,r)}
				{\bf 1}_{(\mathrm{\widetilde{HB}}(a_d))}(S,T)
	\det  (SS^*) e^{-\frac{1}{4\pi^2 }
 \|S\|^2-\frac{1}{8\pi^4 } \|T\|^2}\mathrm{d}S \mathrm{d}T,
\end{multline}
 where $$C(n,r):=
(2\pi)^{nr+\dim_\C (\sym_\C(n,r))}\pi^{nr+2\dim_\C (\sym_\C(n,r))}2^{nr},$$
and $ \|T\|^2 = \tr (TT^*),$ see~(\ref{meumeu}).
Note that in \eqref{mino} we used that for a linear complex $S$, 
if $S_\R$ denotes its associated linear real operator, then
$$| \det  S_\R S_\R^*|^{\frac{1}2} = \det S S^*,$$
and that the support of the latter Gaussian measure
is included in $L(\C^{n},\C^{r})\times \sym_\C(n,r)$, that is, almost surely, $S$ is a complex linear map and $T$ is a complex bilinear map, see Remark \ref{remarkalmostsure}. 

By Fubini and 
since the measure of $T$ is invariant under the symmetries of $\C^{n}$, in the right hand side of~(\ref{mino}),
one can assume that 
$$ \ker S = \C^{n-r}\times \{0\}\subset \C^n$$
in $\mathrm{\widetilde {HB}}(a_d)$. 
Moreover 
by Lemma~\ref{comparaison} and Proposition~\ref{dist}, 
$$\min_{(X,Y)\in (\mathbb S^{2n-2r-1})^2}\|T(X,Y)\|^2
= (n-r)^2\dist^2(\Phi_1(T), \Delta_1),
$$  
where $\Phi_1$ and $\Delta_1$ are defined at the beginning of Section~\ref{sysy}.

Hence, by Lemma~\ref{codim} and Proposition~\ref{burg}, there exist $C,C',C''>0$ (that are independent of the sequence $(a_d)_d$) such that for $d$ large enough,
\begin{multline}\label{tubeestimate}
 \mathbb E\left[{\bf 1}_{(\mathrm{{HB}}(a_d))}|\det  SS^*|^{\frac12} \, \big\vert \, F=0
 \right]
 \\\leq   C
\int_{\substack{(S,T)\in L(\C^n,\C^r)\times \sym_\C(n-r,r)\\
\dist^2(\Phi_1(T),\Delta_1)< C\frac{a_d+h}{2d}\|S\|^2}}
\det  (SS^*) e^{-\frac{1}{4\pi^2 } \|S\|^2-\frac{1}{8\pi^4 } \|T\|^2}\mathrm{d} S\mathrm{d} T
\\ \leq
C'
\int_{S\in L(\C^{n},\C^r)} \det  (SS^*)
 \int_{\rho>0} \rho^{2N(n-r,r)-1}  e^{-\rho^2} \left(\frac{(a_d+h)\|S\|^2}{\rho^2 d}\right)^{3r-2n+3}e^{-\frac{1}{4\pi^2}\|S\|^2}\mathrm{d} S\\
 \leq   C''(\frac{a_d+h}{d})^{3r-2n+2},
\end{multline}
where we identified $T$ and $\Phi_1(T)$ through the isometry $\Phi$ defined in the proof of Lemma~\ref{lembi}.

Thus, by \eqref{KR}, \eqref{reno} and \eqref{tubeestimate}, there exists a constant $C$ such that 

$$\mathbb E 
[
\vol(\hbc_{Z(s)}>-a_d)
]\leq Cd^r\left(\frac{a_d+C}{d}\right)^{3r-2n+3}.$$

Now, since by Lemma~\ref{wirtinger}
$\vol(Z(s))$ is a deterministic quantity $\vol_d$ of order $O(d^r)$. 
In particular, for any $d$ large enough and recalling~(\ref{KR}) and~(\ref{reno}), we obtain
$$ \mathbb E 
[
\density(\hbc_{Z(s)}>-a_d)
]\leq C\left(\frac{a_d+C}{d}\right)^{3r-2n+3},
$$
which is equivalent to $$ \mathbb E 
[
\density(\hbc_{Z(s)}<-a_d)
]\geq 1-C\left(\frac{a_d+C}{d}\right)^{3r-2n+3}.
$$
Hence the result for the holomorphic bisectional curvature.

Similarly,  following the same proof and using Lemma~\ref{codim} for the codimensions of $\Delta_2$ and $\Delta_3$, we obtain
\begin{enumerate}
	\item $ \displaystyle\mathbb E 
[
\vol(\hc >-a_d)
]
\leq C d^r\left(\frac{a_d+C}{d}\right)^{2r-n+1}$
\item $\displaystyle\mathbb E 
[
\vol(\ricci >-a_d)
]
\leq Cr^r
 \left(\frac{a_d+C}{d}\right)^{r(n-r)-(n-r-1)}$,
\end{enumerate}
 which implies the result for the holomorphic sectional curvature and for the Ricci curvature.
 
Finally, for the scalar curvature, one can directly compute in an analogous way:
\begin{eqnarray*}
	\mathbb E 
	[
	\vol(\scal_{Z(s)} >-a_d)
	]
	& \leq & C
	d^r \vol M 
	\int_{
		\|S\|^{-2}\|T\|^2< \frac{a_d+h}{2d}}\det  (SS^*) e^{-\frac{1}{4\pi^2 } \|S\|^2-\frac{1}{4\pi^4 } \|T\|^2}\mathrm{d}S \mathrm{d}T
	\\
	& \leq &
	Cd^r \int_{S} \det  (SS^*) \left(\frac{(a_d+h)\|S\|^2}{ d}\right)^{\frac{1}2r(n-r)(n-r+1)}e^{-\frac{1}{4\pi^2}\|S\|^2}\mathrm{d}S\\
	& \leq & C'd^r (\frac{a_d+C}{d})^{\frac{1}2r(n-r)(n-r+1)}.
\end{eqnarray*}
which implies the result.
\end{proof}
\begin{proof}[Proof of Corollary~\ref{coro2}]
Let us prove the corollary for the holomorphic bisectional curvature; the proof for the other curvatures is the same. By Theorem~\ref{theorem1},
	$$\forall d\gg1, \ 
	\mathbb E_{\mu_d}	[
	\density(\hbc_{Z(s)}<-a_d)
	]
	\geq 1-  C\left(\frac{a_d+C}{d}\right)^{3r-2n+2}.$$
	Since for any $\alpha>0$, 
	$$	\mathbb E_{\mu_d}	[
	\density(\hbc_{Z(s)}<-a_d)
	]\leq \alpha \mu_d [\density(\hbc_{Z(s)}<-a_d)<\alpha]
	+\mu_d [\density(\hbc_{Z(s)}<-a_d)\geq \alpha],$$
	we obtain
	$$
	\mu_d [\density(\hbc_{Z(s)}<-a_d)\geq \alpha]\geq 1-
	\frac{1}{\alpha}	C\left(\frac{a_d+C}{d}\right)^{3r-2n+2} .$$
	We conclude choosing $a_d d= d^{\epsilon}$ and
	$\alpha= \alpha_d =1- d^{-\eta}.$
\end{proof}
\begin{proof}[ Proof of Corollary~\ref{almost}]
	All the assertions are proven in the same way, so we write it only for the scalar curvature.
	For any $d$ large enough
	and any $s=(s_d)_d\in \mathcal H$, 
	let $$\forall d\geq 1, \ Y_d(s) := 1-\density (\scal_{Z(s_d)}<-1).$$
	Then, by Theorem~\ref{theorem1},
	$$0\leq  \int_{s\in \mathcal H}Y_d(s) d\mu(s) = O(d^{-\frac{1}2r(n-r)(n-r+1)}),$$
	so that 
	$$\int_{s\in  \mathcal H} \sum_{d=1}^\infty Y_d(s)d\mu(s)= \sum_{d=1}^\infty \int_{s\in \mathcal H}Y_d(s) d\mu(s)  < \infty,$$
	so that almost surely, $Y_d(s)\to_{d\to \infty} 0.$
\end{proof}
\begin{remark} In fact, one can prove a better estimate. For the scalar curvature, 
	let $$ 0<\epsilon <\frac{1}2 r(n-r)(n-r+1)-1.$$
	Then, 
	$$0\leq  \int_{s\in \mathcal H}d^{\epsilon}Y_d(s) d\mu(s) = O(d^{\epsilon -\frac{1}2r(n-r)(n-r+1)}),$$
	and the same argument as before shows that, almost surely,
	$$ \forall d\gg 1, \ 1-\density (\scal_{Z(s_d)}<-1) = o(d^{-\epsilon}).$$
\end{remark}

For the reader's convenience, Proposition~\ref{WKC} below provides a proof of the asymptotic version of Proposition~\ref{wirtinger} using the Kac--Rice formula.
\begin{proposition}\label{WKC}
	Under the hypotheses of Theorem~\ref{theorem1}, 
	$$ \mathbb E (\vol Z(s))
	\sim_{d\to \infty}
	\frac{n!}{(n-r)!}d^r \vol (M).$$
\end{proposition}
\begin{proof}
	By the proof of Theorem~\ref{theorem1}, forgetting the superfluous $\pi$ in the BF-measure, 
	$$
	\mathbb E (\vol Z(s))\sim_d d^r \vol (M)
	\mathbb E_{\mathrm{BF}}\left[\det  (SS^*)
	\right]\rho^{\mathrm{BF}}_{F}(0).
	$$ 
The random matrix $S S^*$ is a Wishart complex matrix and by~\cite[Theorem 3.1]{nagar2011expectations},
	$$ 
	\mathbb E_{BF}
	\left[\det  S S^*) 
	\right] = \frac{n!}{(n-r)!}\det (\frac{1}{n}\Sigma),
	$$
	where, writing $S_\C= (s_{ij})_{1\leq i\leq r, 1\leq j\leq n}\in M_{n,r}(\C)$, $$\Sigma := \mathbb E (S_\C S_\C^*)=
	\left(\sum_{k=1}^n
	\mathbb E  s_{ik}\overline{s_{jk}}\right)_{1\leq i,j\leq r}= \left(\sum_{k=1}^n
	\mathbb E  2\pi \delta_{i,j}\right)_{1\leq i,j\leq r}= 
	2\pi n I_r.$$
	Now, $\rho^{\mathrm{BF}}_F(0)= \frac{1}{(2\pi)^r},$ so that
	$$ \mathbb E_{\mathrm{BF}} (\vol Z(s))\sim_{d\to \infty}
	\frac{n!}{(n-r)!}d^r \vol (M),$$
	which is indeed the asymptotics of $\vol_d$ given by Lemma~\ref{wirtinger}.
\end{proof}

\bibliographystyle{amsplain}
\bibliography{courbure.bib}

\end{document}